\documentclass[11pt,english]{amsart}
\usepackage{amsmath, amsthm, latexsym, amssymb}
\usepackage{amsfonts, verbatim}
\usepackage{xypic}
\usepackage{epsfig}
\input xy
\xyoption{all}

\newcommand{\shrinkmargins}[1]{
  \addtolength{\textheight}{#1\topmargin}
  \addtolength{\textheight}{#1\topmargin}
  \addtolength{\textwidth}{#1\oddsidemargin}
  \addtolength{\textwidth}{#1\evensidemargin}
  \addtolength{\topmargin}{-#1\topmargin}
  \addtolength{\oddsidemargin}{-#1\oddsidemargin}
  \addtolength{\evensidemargin}{-#1\evensidemargin}
  }

\shrinkmargins{.7}

\newcommand{\field}[1]{\mathbb{#1}}

\newcommand{\Q}{\field{Q}}

\newcommand{\Z}{\field{Z}}

\newcommand{\C}{\field{C}}

\newcommand{\beq}{\begin{displaymath}}
\newcommand{\eeq}{\end{displaymath}}
\newcommand{\beqn}{\begin{equation}}
\newcommand{\eeqn}{\end{equation}}

\theoremstyle{plain}
\newtheorem{theorem}{Theorem}[section]
\newtheorem{proposition}[theorem]{Proposition}
\newtheorem{corollary}[theorem]{Corollary}
\newtheorem{lemma}[theorem]{Lemma}

\theoremstyle{definition}
\newtheorem{definition}[theorem]{Definition}

\newtheorem*{teo}{Theorem}
\newtheorem*{coro}{Corollary}

\theoremstyle{remark}
\newtheorem{remark}[theorem]{Remark}

\newtheorem{question}[theorem]{Question}

\title[Modular forms arising from totally real cubic fields]{A space of weight 1 modular forms attached to totally real cubic number fields}

\author{Guillermo Mantilla-Soler \\ }
\date{}

\begin{document}

\maketitle

\begin{abstract}
 Let $d$ be a positive fundamental discriminant, and let $\mathcal{C}_{d}$ be the set of isomorphism classes of cubic number fields of discriminant $d$. For each $K \in \mathcal{C}_{d}$, we construct a weight 1 modular form $f_{K}$ with level $3^{\pm 1}d$ and nebentypus $\left( \frac{-3^{\pm 1}d}{\cdot} \right)$. We show that the form $f_{K}$ completely determines the field $K$. Moreover, we show that $\{f_{K} : K \in \mathcal{C}_{d}\}$ is a linearly independent set.

\end{abstract}
 
\section{Introduction}

Let $\delta$ be a fundamental discriminant, and let $\mathcal{C}_{\delta}$ be the set of isomorphism classes of cubic number fields of discriminant $\delta$. Recall that an integer is called {\it fundamental discriminant} if it is equal to the discriminant of a quadratic number field. Let $N$ be a positive integer and let $\epsilon$ be a Dirichlet character modulo $N$. Let $\displaystyle \mathcal{M}_{1}\left(\Gamma_{0}(N),  \epsilon \right)$ be the space of weight $1$ modular forms of level $N$ and nebentypus $\epsilon$. Suppose $\mathcal{C}_{\delta} \neq \emptyset$ and that $\delta <0$, i.e. cubic fields with at least one complex place. Then, associated to each $K \in \mathcal{C}_{\delta}$ theres is a weight 1 modular form \[\mathfrak{f}_{K} \in \displaystyle \mathcal{M}_{1}\left(\Gamma_{0}(|\delta|), \left( \frac{\delta}{\cdot} \right)\right)\]  such that:
\begin{enumerate}
\item the map $K  \mapsto \mathfrak{f}_{K}$ is injective.

\item the set $\{\mathfrak{f}_{K} : K \in \mathcal{C}_{\delta}\}$ is a linearly independent subset of $\displaystyle \mathcal{M}_{1}\left(\Gamma_{0}(|\delta|), \left( \frac{\delta}{\cdot} \right)\right)$.

\end{enumerate}

The above follows from a particular case of Weil-Langlands converse theorem. See \S \ref{clasica} for details. If instead of considering cubic fields ramified at infinity we consider totally real cubic fields, then it is not possible to apply Weil-Langlands to produce modular forms. The point is that in the totally real case the Galois representations involved are even. In this paper we provide for totally real cubic fields of fundamental discriminant an alternative construction of weight 1 modular forms satisfying properties (1) and (2) above. Given $K$ a cubic field with positive fundamental discriminant we define $\Theta(K):=f_{K}$ to be the form in Corollary \ref{defitheta}. For the value $d_{3}$ see Definition \ref{reflex} below. 
The following are the main resutls of this paper: 
\begin{teo}[cf. Theorem \ref{1to1}] {\it Let $d$ be a positive fundamental discriminant, and suppose that $\mathcal{C}_{d}\neq \emptyset$. Then, 
\begin{displaymath}
\begin{array}{cccc}
 \Theta: & \mathcal{C}_{d} &\rightarrow& \mathcal{M}_{1}\left(\Gamma_{0}(|d_{3}|), \left( \frac{d_{3}}{\cdot} \right)\right) \\  
           & K & \mapsto & f_{K}
\end{array}
\end{displaymath}
is injective.}
\end{teo}

\begin{coro}[cf. Corollary \ref{base}]
{\it Let $d$ be a positive fundamental discriminant, and suppose that $\mathcal{C}_{d}\neq \emptyset$. Then $\Theta(\mathcal{C}_{d})$ is a linearly independent subset of $\mathcal{M}_{1}\left(\Gamma_{0}(|d_{3}|), \left( \frac{d_{3}}{\cdot} \right)\right)$.}
\end{coro}

 \section{Weight 1 forms attached to non-totally real cubic fields.}\label{clasica} 
We briefly recall how from Weil-Langlands for $S_{3}$ one can obtain a family of linearly independent weight 1 modular forms parametrized by the set of isomorphism classes of cubic fields of a fixed negative discriminant. \\

Let $L$ be a cubic number with discriminant $\delta<0$, which is fundamental, and  let $\widetilde{L}$ be its Galois closure.

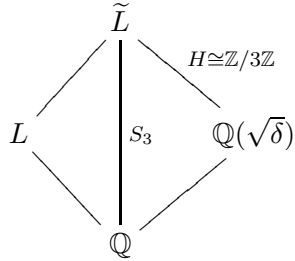
\[\xymatrix{  & \widetilde{L}\ar@{-}[dd]^{S_{3}}\ar@{-}[ld] \ar@{-}[rd]^{H \cong \Z/3\Z}    & \\
             L\ar@{-}[rd] & &  \Q(\sqrt{\delta})\ar@{-}[ld] \\
            & \Q &  }\]

By Identifying $S_{3}$ with ${\rm Gal}(\widetilde{L}/\Q)$, and considering the irreducible 2-dimensional representation of $S_{3}$, one obtains an irreducible dihedral representation of $\rho_{L}: {\rm Gal}(\overline{\Q}/\Q) \to {\rm GL}_{2}(\C)$. Such a representation is induced by a non-trivial representation of $H$. Since $d$ is a fundamental discriminat $\widetilde{L}$ is contained in the Hilbert class field of  $\Q(\sqrt{\delta})$. Thus, by the conductor formula for induced representations, we have that $\rho_{L}$ has conductor $|\delta|$. From the above diagram we see that $\det(\rho_{L})$ is the inflation of the non-trivial character of ${\rm Gal}(\Q(\sqrt{\delta})/\Q)$. In other words, for all prime $p$ not dividing $d$ we have that $\det(\rho_{L}({\rm Frob}_{p}))=\left( \frac{\delta}{p} \right)$. In particular, $\rho_{L}$ is an odd representation. Let $\displaystyle L(s,\rho_{L} )=\sum_{n=1} \frac{a_{n}}{n^{s}}$ be the Artin $L$-series attached to $\rho_{L}$ and let $\displaystyle \mathfrak{f}_{L}=\sum_{n=1}a_{n}q^{n}$ be the $q$-expansion of $L(s, \rho_{L})$.  It follows from \cite[Theorem 1]{serre} that \[\mathfrak{f}_{L} \in \mathcal{M}_{1}\left(\Gamma_{0}(|\delta|), \left( \frac{\delta}{\cdot} \right) \right).\]  Moreover, each $\mathfrak{f}_{L}$ is a normalized cuspidal eigenform.

\begin{remark}\label{Aeq}
Notice that $\zeta_{\widetilde{L}}(s)=L^{2}(s, \rho_{L})\zeta_{\Q(\sqrt{\delta})}.$
\end{remark}

\begin{theorem}\label{negativebase}
The following is injective:
\begin{displaymath}
\begin{array}{cccc}
 \Phi: & \mathcal{C}_{\delta} &\rightarrow& \mathcal{M}_{1}\left(\Gamma_{0}(|\delta|), \left( \frac{\delta}{\cdot} \right)\right) \\  
           & L & \mapsto & \mathfrak{f}_{L}
\end{array} 
\end{displaymath}

Furthermore, $\Phi(\mathcal{C}_{\delta})$ is a linearly independent subset of $\mathcal{M}_{1}\left(\Gamma_{0}(|\delta|), \left( \frac{\delta}{\cdot} \right)\right)$.
\end{theorem}

\begin{proof}
Let $L, L_{1}$ be cubic fields of discriminant $\delta$. Then  \[\Phi(L)=\Phi(L_{1}) \iff L(s,\rho_{L} )= L(s,\rho_{L_{1}} ).\]  It follows, see Remark \ref{Aeq}, that the fields $\widetilde{L}$ and $\widetilde{L_{1}}$ are Arithmetically equivalent (See \cite{Perlis}). Since $L$ and $\widetilde{L}$ are both Galois over $\Q$ we have that  $\widetilde{L} \cong \widetilde{L_{1}}$ and consequently $L \cong L_{1}.$ Since a set of normalized cuspidal eigenforms non containing pairs of linearly dependent forms is a linearly independent set, the linearly independence of  $\Phi(\mathcal{C}_{d})$ follows from the injectivity of $\Phi$. 
\end{proof}

The construction above gives a natural subspace of \[\mathcal{M}_{1}\left(\Gamma_{0}(|\delta|), \left( \frac{\delta}{\cdot} \right)\right),\] which has a canonical basis indexed by cubic fields of discriminant $\delta$.  
\begin{definition}\label{cubespace}
We called this space $V_{\delta} :=\mathrm{Span}_{\C}\Phi(\mathcal{C}_{\delta})$ the {\it $\delta$-cubic space}.
\end{definition}

\begin{remark}
Notice that $V_{\delta}$ is actually a Hecke invariant subspace of the cuspidal subspace of $\mathcal{M}_{1}\left(\Gamma_{0}(|\delta|), \left( \frac{\delta}{\cdot} \right)\right).$
\end{remark}

\section{Totally real cubic fields}   

The main goal of this paper is to exhibit a canonical subspace of a space of weight 1 modular forms that is parametrized by the set of isomorphism classes of cubic fields of a fixed fundamental discriminant.  In case that the fields have ramification at infinity, the construction is classic and it is explained in \S \ref{clasica}. In this section we show how to construct such a subspace for cubic fields with no ramification at infinity i.e., totally real cubic fields.

\subsection{Main construction:}

Let $K$ be a number field and let $O^{0}_K \subset K$ be the set of integral elements with zero trace. The 
map
\begin{displaymath}
\begin{array}{cccc}
\mathrm{t}^{0}_K : & O^{0}_K  &  \rightarrow & \Z  \\  & x & \mapsto &
\mathrm{tr}_{K/\Q}(x^2)
\end{array}
\end{displaymath}
defines a quadratic form with corresponding bilinear form
\[T_K(x,y ) = \mathrm{tr}_{K/\Q}(xy)\Big |_{O^{0}_K}.\]

\begin{definition}\label{reflex}
Let $d$ be a non-zero integer. The {\it $3$-reflection}  $d_{3}$ of $d$ is the integer defined by $\displaystyle d_{3} := -\frac{3d}{\mathrm{gcd}(3,d)^2}$.
\end{definition}

\begin{proposition}
Let $K$ be a cubic field with fundamental discriminant $d$. Then, \[\mathrm{t}_K :=
\frac{\mathrm{t}^{0}_{K}}{2\mathrm{gcd}(3,d)}\] is a primitive
integral binary quadratic form of discriminant $d_{3}$. 
\end{proposition}

\begin{proof}
If $3 \nmid d$ this is precisely \cite[Corollary 5.4]{Manti}.  Otherwise, the result follows from  \cite[Theorem 6.4]{Manti}.
\end{proof}

It turns out that for a totally real cubic field $K$ the form $\mathrm{t}_K$ completely characterizes the field.

\begin{theorem}\label{mio}\cite[Theorem 6.5]{Manti}
Let $d$ be a positive fundamental discriminant and let $K$ and $L$ be two elements in $\mathcal{C}_{d}$. Then, \[K \cong L \iff \mathrm{t}_K  \cong \mathrm{t}_L,\] where the second $\cong$ refers to equivalence under the natural GL$_{2}(\Z)$-action.
\end{theorem}

\begin{definition}
Let $K$ be a totally real cubic field with fundamental discriminant $d$. Let \[f_{K}(z):= \sum_{x \in O_{K}^{0}} q^{\mathrm{t}_{K}(x)}\] where $q=e^{2\pi iz}$, and $z$ lies in the upper half plane.
\end{definition}

We have associated to $K$ a theta series given by the quadratic form $\mathrm{t}_{K}$. The following classic result of Schoeneberg \cite{sch} gives us some information about $f_{K}$. 

\begin{theorem}[Schoenberg]
Let $Q$ be a positive definite, integral quadratic form of dimension $2k$, and discriminant $\Delta$. Let \[\theta(Q,z):=\sum_{x \in \Z^{2k}} q^{Q(x)}\] be the theta series associated to $Q$. Then, 
\[\theta(Q,z) \in \mathcal{M}_{k}\left(\Gamma_{0}(|\Delta|),  \left(\frac{\Delta}{\cdot} \right) \right).\] 
\end{theorem}
See \cite[Theorem 20-20$^{+}$]{ogg}.

\begin{corollary}\label{defitheta}
Let $K$ be a totally real cubic field with fundamental discriminant $d$. Then, \[f_{K} \in  \mathcal{M}_{1}\left(\Gamma_{0}(|d_{3}|), \left( \frac{d_{3}}{\cdot} \right)\right)\]
\end{corollary}

\begin{proof} 
Since $\mathrm{t}_{K}$ is a binary form of discriminant $d_{3}$ the result is a particular case of Schoenberg's Theorem. 
\end{proof}

Let $d$ be a positive fundamental discriminant, and recall that $\mathcal{C}_{d}$ denotes the set of isomorphism classes of cubic number fields with discriminant $d$.  The following is the key result in the construction of a basis for the space of modular forms mentioned at the beginning of this section.   
\begin{theorem}\label{1to1}
  
The map 
\begin{displaymath}
\begin{array}{cccc}
 \Theta: & \mathcal{C}_{d} &\rightarrow& \mathcal{M}_{1}\left(\Gamma_{0}(|d_{3}|), \left( \frac{d_{3}}{\cdot} \right)\right) \\  
           & K & \mapsto & f_{K}
\end{array}
\end{displaymath}
is injective.
\end{theorem}

\begin{proof}
Suppose that $\Theta(K)=\Theta(L)$. Since binary positive deÞnite forms are determined up to
integral equivalence by their theta series (see \cite{watson}) we have that the  forms $\mathrm{t}_{K}$ and $\mathrm{t}_{L}$ are equivalent. Thus thanks to Theorem \ref{mio} we have that $K$ and $L$ are isomorphic.  \end{proof}

We add a proof of the following simple fact for the lack of a suitable reference.

\begin{lemma}\label{primitive}Let $Q$ and $Q_{1}$ be two integral binary quadratic forms of the same discriminant. If there exists a prime $p$ which is represented by $Q$ and $Q_{1}$, then the forms are equivalent. 
\end{lemma}

\begin{proof}
We may assume that $Q(x,y)=px^2+bxy+cy^2$, and that $Q_{1}(x,y)=px^2+b_{1}xy+c_{1}y^2$ where $0 \leq b,b_{1} \leq p$. Since $b^2-4pc=b^{2}_{1}-4pc_{1}$ we have that 
$2p \mid (b+\epsilon b_{1})$ where $\epsilon \in \{1,-1\}$. If $b+\epsilon b_{1} \neq 0$ then $2p \leq | b+\epsilon b_{1}| \leq b+b_1 \leq p +p =2p$, hence $b=p=b_1$ and $c=c_1$. If $b+\epsilon b_{1} = 0$ then $b=\pm b_1$, and since they are both non-negative $b=b_1$ and again $c=c_1$. In any case $Q=Q_1$.
\end{proof} 
%\begin{proposition}\label{hayp}
%Let $Q$ be a primitive, positive definite, integral binary quadratic form of discriminant $\Delta$. Then, there are primes represented by $Q$.
%\end{proposition}

%\begin{proof}
%Let $Cl(\Delta)$ be the set of SL$_{2}(\Z)$ classes of primitive, positive definite, integral binary quadratic forms of discriminant $\Delta$, and let $Cl(\Q(\sqrt{-\Delta}))$ be ideal class group of $\Q(\sqrt{-\Delta})$. Recall that there is a bijection $\Gamma_{\Delta} :  Cl(\Delta) \to Cl(\Q(\sqrt{-\Delta}))$, which in fact is a group isomorphism for the group structure on $Cl(\Delta)$ given by Gauss's composition (See \cite{buell} for details). From the definition of $\Gamma_{\Delta}$, and since conjugation acts as taking inverses on $Cl(\Q(\sqrt{-\Delta}))$, it follows that a prime $p$ is represented by $Q$ if and only if $p$ is the norm of a prime ideal $P \in \Gamma_{\Delta}([Q])$. The result follows by the above proposition with $K=\Q(\sqrt{-\Delta})$ and $[I]=\Gamma_{\Delta}([Q]).$
%\end{proof}

It is a classic result of Hecke that given a set of  inequivalent positive binary quadratic forms with a fixed fundamental discriminant, their theta series are linearly independent. From Hecke's result and Theorem \ref{1to1} we obtain:

\begin{corollary}\label{base}
Let $d$ be a positive fundamental discriminant. Then $\Theta(\mathcal{C}_{d})$ is a linearly independent subset of $\mathcal{M}_{1}\left(\Gamma_{0}(|d_{3}|), \left( \frac{d_{3}}{\cdot} \right)\right)$.
\end{corollary}

\begin{proof} We may assume that $\Theta(\mathcal{C}_{d})$ is non-empty.
Let $K_{1},..,K_{n}$ be a set of representatives of $\mathcal{C}_{d}$, and let $f_{i}:=\Theta(K_{i})$ for $i=1,...,n$. Suppose there are complex numbers $\lambda_{1},...,\lambda_{n}$ such that 
\[\lambda_{1}f_{1}+...+\lambda_{n}f_{n}=0.\]
Let $p$ be a prime represented by $\mathrm{t}_{K_{1}}$, which exists thanks to the Dirichlet-Weber Theorem \cite[pg 72]{nark}. Since $\Theta$ is injective it follows from Lemma \ref{primitive}  that $p$ is not represented by $\mathrm{t}_{K_{i}}$ for any  $i=2,...,n$. Thus the coefficient of $q^{p}$ in $f_{i}$ is equal to zero for all $i=2,...,n$. In particular, the coefficient of  $q^{p}$ in $\lambda_{1}f_{1}+...+\lambda_{n}f_{n}$ is $\lambda_{1}f_{1}(p)$ where $f_{1}(p)$ is the coefficient of  $q^{p}$ in $f_1$. Since $p$ is represented by $\mathrm{t}_{K_{1}}$ we have that $f_{1}(p) \neq 0$ hence $\lambda_1=0$. Since the order of the $\lambda$'s is irrelevant we deduce that $\lambda_{i}=0$ for all $i=1,...,n$.
\end{proof}

Let $d$ be a positive fundamental discriminant with $\mathcal{C}_{d} \neq \emptyset$. We have exhibited a natural subspace of \[\mathcal{M}_{1}\left(\Gamma_{0}(|d_{3}|), \left( \frac{d_{3}}{\cdot} \right)\right),\] which has a canonical basis indexed by cubic fields of discriminant $d$.  \begin{definition}\label{positivecubespace} We called this space $V^{+}_{d} :=\mathrm{Span}_{\C}\Theta(\mathcal{C}_{d})$ the {\it positive $d$-cubic space}.\end{definition}

\section{Explicit Scholz reflection principle}

In this section we make some speculations as to how the results mentioned in this paper can be used to give an explicit version of the Scholz reflection principle on cubic fields.\\

Given a fundamental discriminant $\delta$ we define it's 3-rank to be the 3-rank of the ideal class group of the quadratic field of discriminant $\delta$ i.e., \[r_{3}(\delta):=\dim_{\mathbb{F}_{3}}(Cl(\Q(\sqrt{\delta}))
\otimes_{\Z} \mathbb{F}_3).\]
Let $d$ be a positive fundamental discriminant.  The classic Scholz reflection principle says that \[ r_{3}(d) \leq r_{3}(d_{3})\leq  r_{3}(d)+1.\]
This can be stated in terms of subspaces $V_{d}^{+}$ and $V_{d_{3}}$ (see Definitions \ref{positivecubespace} and \ref{cubespace}) of $\displaystyle \mathcal{M}_{1}\left(\Gamma_{0}(|d_{3}|), \left( \frac{d_{3}}{\cdot} \right)\right).$

\begin{proposition} Let $d$ be a positive fundamental discriminant and let $V_{d}^{+}$ (resp. $V_{d_{3}}$) be positive $d$-cubic subspace (reps. $d_{3}$-cubic subspace) of $\displaystyle \mathcal{M}_{1}\left(\Gamma_{0}(|d_{3}|), \left( \frac{d_{3}}{\cdot} \right)\right).$ Then, 

\[\dim_{\C}(V_{d_{3}})=\dim_{\C}(V_{d}) \ or \ \dim_{\C}(V_{d_{3}})=3\dim_{\C}(V_{d})+1.\]

\end{proposition}

\begin{proof}  Let $\delta$ be a fundamental discriminant. By a result of Hasse (see \cite{hasse})  the number of isomorphism classes of cubic fields of discriminant $\delta$ is equal to $\frac{(3^{r_{3}(\delta)}-1)}{2}$. Hence the result is equivalent to the Scholz reflection principle thanks to Theorem \ref{negativebase} and Corollary \ref{base}.
\end{proof}

It follows that the space $V_{d}$ embeds non-canonically into $V_{d_{3}}$. In fact, as seen in the proof above, the existence of such an embedding is equivalent to the left inequality  in the Scholz reflection principle. Suppose that $d$ is such that $\mathcal{C}_{d} \neq \emptyset$. It is natural interesting to ask if we can make a choice for the embedding.

\begin{question}
Is there a canonical embedding $\displaystyle V_{d} \to V_{d_{3}}?$ 
 \end{question}

A stronger version of the above, which can be interpreted as an  Explicit Scholz reflection principle, is the following:

\begin{question}\label{expliScholz} 
Is there a canonical injection $\Theta(\mathcal{C}_{d}) \to \Phi(\mathcal{C}_{d_{3}})?$
\end{question}

If Question \ref{expliScholz} had a positive answer that would mean, thanks to Theorems \ref{negativebase}, \ref{base}, that for every cubic field $K$ of discriminant $d$ there would exists a canonical way to pick a unique cubic field $K_{3}$ of discriminant $d_{3}$.

%\section*{Acknowledgements}

%I thank Bill Casselman, Julia Gordon, Lior Silberman and all the participants of the Number theory and Automorphic forms study seminar 2010/2011 at UBC since it was their seminar what inspired me to write this paper. I also thank Sujatha and Nike Vatsal for their advice and their comments on a previous version of this paper.

\end{document}